\pdfoutput=1
\documentclass[12pt]{amsart}

\usepackage[utf8]{inputenc}
\usepackage{amssymb}
\usepackage{amsfonts}
\usepackage{color}

\usepackage[backend=bibtex, url=false, doi=false, maxnames=5]{biblatex}
\renewbibmacro{in:}{}
\AtEveryBibitem{
 \clearlist{address}
 \clearfield{date}
 \clearfield{isbn}
 \clearfield{issn}
 \clearlist{location}
 \clearfield{month}
 \clearfield{series}
 \clearfield{language}

 \ifentrytype{book}{}{
  \clearlist{publisher}
  \clearname{editor}
 }
}
\renewbibmacro*{volume+number+eid}{%
  \bfseries
  \printfield{volume}%
  \mdseries
  \setunit*{\addnbspace}
  \printfield{number}%
  \setunit{\addcomma\space}%
  \printfield{eid}}
\DeclareFieldFormat[article]{number}{\mkbibparens{#1}}

\newbibmacro{string+doiurl}[1]{%
	\iffieldundef{doi}{%
		\iffieldundef{url}{%
			#1%
		}{\href{\thefield{url}}{#1}}%
	}{\href{http://dx.doi.org/\thefield{doi}}{#1}}%
}

\DeclareFieldFormat{title}{\usebibmacro{string+doiurl}{\mkbibemph{#1}}}
\DeclareFieldFormat[article,incollection,thesis]{title}{\usebibmacro{string+doiurl}{\mkbibquote{#1}}}

\usepackage{paralist}

\usepackage{microtype}
\usepackage{hyperref}
\definecolor{darkblue}{RGB}{0,0,160}
\hypersetup{
	colorlinks,%
	citecolor=black,%
	filecolor=black,%
	linkcolor=darkblue,%
	urlcolor=darkblue
}

\usepackage{amsthm}

\usepackage[
text={440pt,575pt},
headheight=9pt,
centering
]{geometry}

\newcommand{\excise}[1]{}

\newtheorem{thm}{Theorem}
\newtheorem{lemma}[thm]{Lemma}

\newtheorem{cor}[thm]{Corollary}
\newtheorem{prop}[thm]{Proposition}

\newtheorem{alg}[thm]{Algorithm}

\theoremstyle{definition}
\newtheorem{example}[thm]{Example}
\newtheorem{remark}[thm]{Remark}
\newtheorem{defn}[thm]{Definition}

\numberwithin{equation}{section}

\newcommand{\ring}[1]{\ensuremath{\mathbb{#1}}}

\renewcommand{\>}{\rangle}
\newcommand{\<}{\langle}
\newcommand\CC{\ring{C}}
\newcommand\NN{\ring{N}}
\newcommand\QQ{\ring{Q}}
\newcommand\RR{\ring{R}}
\newcommand\ZZ{\ring{Z}}
\newcommand\kk{\mathbb{K}}
\newcommand\cB{{\mathcal B}}

\newcommand\xb{\mathbf{x}}
\newcommand\yb{\mathbf{y}}

\newcommand{\QLaurx}{\QQ[x^\pm_1, \dots, x^\pm_n]}
\newcommand{\kLaurx}{\kk[x^\pm_1, \dots, x^\pm_n]}
\newcommand{\QLaurym}{\QQ[y^\pm_1, \dots, y^\pm_m]}
\newcommand{\kLaurym}{\kk[y^\pm_1, \dots, y^\pm_m]}

\def\ol#1{{\overline {#1}}}

\newcommand\set[1]{\{#1\}}
\newcommand{\with}{\colon}
\newcommand{\units}{\times}

\DeclareMathOperator\Id{Id} 
\DeclareMathOperator\rank{rank} 
\DeclareMathOperator{\Span}{span} 
\newcommand{\defas}{\mathrel{\mathop{:}}=}   
\DeclareMathOperator{\ini}{in} 

\DeclareMathOperator{\binn}{Bin} 
\newcommand{\bin}[1]{\binn(#1)}

\begin{document}

\title{Finding binomials in polynomial ideals}

\def\urladdrname{{\itshape Website}}

\author{Anders Jensen}
\address{Aarhus Universitet\\ Aarhus, Denmark}
\urladdr{\url{http://home.math.au.dk/jensen/}}

\author{Thomas Kahle}
\address{Otto-von-Guericke Universität\\ Magdeburg, Germany}
\urladdr{\url{http://www.thomas-kahle.de}}

\author{Lukas Katthän}
\address{Goethe-Universität Frankfurt\\ Frankfurt, Germany}
\urladdr{\url{http://www.math.uni-frankfurt.de/~katthaen/}}

\date{\today}


%



\begin{abstract}
We describe an algorithm which finds binomials in a given ideal
$I\subset\QQ[x_1,\dots,x_n]$ and in particular decides whether
binomials exist in $I$ at all.  Binomials in polynomial ideals can be
well hidden.  For example, the lowest degree of a binomial cannot be
bounded as a function of the number of indeterminates, the degree of
the generators, or the Castelnuovo--Mumford regularity.
We approach the detection problem by reduction to the Artinian case
using tropical geometry.  The Artinian case is solved with algorithms
from computational number theory.
\end{abstract}

\maketitle


\section{Introduction}
\label{sec:introduction}

\noindent The goal of this paper is to prove the following result.
\begin{thm}\label{t:main}
There is a deterministic algorithm that, given generators, decides
whether an ideal in the polynomial ring $\QQ[x_1,\dots,x_n]$ contains
nonzero binomials.
\end{thm}

Theorem~\ref{t:main} answers a fundamental question in computational
algebra, but we envision that it will also be useful for applications.
To name just a few, when implementing mesoprimary decomposition of
binomial ideals~\cite{kahle11mesoprimary}, a test for binomials is
necessary.  In the theory of retractions of polytopal algebras,
\cite[Conjecture~B]{bruns2002polytopal} is connected to the existence
of binomials and monomials in the kernels of certain maps (albeit
after a graded automorphism of the ambient ring).
In~\cite{sontag2014technique}, Sontag argues that polynomials with few
terms in an ideal yield the best restrictions on the possible sign
patterns of changes that a steady state of a chemical reaction network
can undergo under perturbation.
Theorem~\ref{t:main} can also be seen as a first step to the broader
problem of deciding whether an ideal contains a sparse polynomial, or
finding the sparsest polynomial.  For example, Jürgen Herzog suggested
the problem of determining the length of the shortest polynomial in a
standard determinantal ideal.

It does not seem possible to prove Theorem~\ref{t:main} by standard
arguments using Gröbner bases.  For example the ideal
$\<x^2+x+1\> \subset \QQ[x]$ contains $x^3-1$, but its generator,
trivially, is a universal Gröbner basis.  Moreover, the lowest-degree
binomials in an ideal need not satisfy a general upper degree bound in
terms of common invariants such as Castelnuovo--Mumford regularity or
primary decomposition.
\begin{example}\label{e:deepBinomials}
For any $n\in\NN$, let
$I = \<(x-z)^2, nx - y - (n-1)z)\> \subset \QQ[x,y,z]$.  The
Castelnuovo--Mumford regularity of $I$ is $2$ and it is primary over
$\<x-z, y-z\>$.  The binomial $x^n-y z^{n-1}$ is contained in $I$
because an elementary computation shows that
\[ x^n-y z^{n-1} = \sum_{k=0}^{n-2} (n-k-1)x^k z^{n-k-2} (x-z)^2 +
z^{n-1}(nx - y - (n-1)z) \in I.\]
There is no binomial of degree less than $n$ in~$I$.  To see this,
consider the differential operators
$D_1 = \partial_{x} + n\partial_{y}$ and
$D_2 = (1-n)\partial_{y} + \partial_{z}$.  Any element $f \in I$
satisfies $f(1,1,1) = 0,\ (D_1 f)(1,1,1) = 0$ and $(D_2 f)(1,1,1) = 0$
as both generators have this property.  Assume that $I$ contains the
binomial $f = x^u-\lambda y^v$.  First, note that $f(1,1,1) = 0$
implies that $\lambda = 1$.  Further, $(D_1 f)(1,1,1) = 0$ and
$(D_2 f)(1,1,1) = 0$ give two linear conditions on the vector $u - v$,
which imply that $u - v = m (n, -1, 1-n)$ for some $m \in \ZZ$.  By
exchanging $u$ and $v$ we may assume that $m > 0$, so it follows that
$f = x^{m n} - y^m z^{m (n-1)}$.  In particular, there is no binomial
of degree less than $n$ in $I$.
\end{example}

Our approach to Theorem~\ref{t:main} can be summarized as follows.
Given an ideal $I\subset \QQ[x_1,\dots,x_n]$, we pass to its Laurent
extension $J = I\QLaurx$, which contains binomials if and only if $I$
contains binomials (Lemma~\ref{l:saturate}).  We then show in
Section~\ref{s:reduceToArtinian} that there exists an ideal $J' \subseteq \QLaurym$ such
that $\QLaurym/J'$ is Artinian, and the sets of binomials in $J$ and $J'$ can easily be computed from each other
(Proposition~\ref{p:binproject} and Theorem~\ref{t:reduceArt}).  This
reduction is achieved by means of tropical geometry.  The Artinian case is easier since the (images of
the) indeterminates $y_1,\dots, y_m$ in $\QLaurym/J'$ have matrix
representations that commute.  This leads to the constructive
membership problem for commutative matrix
(semi)groups~\cite{commmatrix}, which is already solved (see
\cite{halava1997decidableSurvey} for a survey).  The completed
algorithm appears as Algorithm~\ref{alg:complete} in
Section~\ref{sec:complete_algorithm}.

\subsection*{Related work and variations of the problem}
The question whether an ideal is a binomial ideal, that is, whether it
can be generated by binomials alone, can be decided by computing a
reduced Gröbner basis~\cite[Corollary~1.2]{eisenbud96:_binom_ideal}.
In the case of a homogeneous ideal one can even do it with linear
algebra only~\cite[Proposition~3.7]{conradi2015detecting}.

Moreover, deciding for given monomials $\xb^u$ and~$\xb^v$ whether
there exists some scalar $\lambda$ such that $\xb^u-\lambda \xb^v$ is
contained in a given ideal~$I$ is also not too difficult.  For this
problem, it suffices to compute the unique normal forms of $\xb^u$ and
$\xb^v$ modulo a Gröbner basis of $I$ and check whether they are
scalar multiples of each other.  Using this observation, one can
decide whether $I$ contains a binomial of a given degree by brute
force.  However, this approach cannot be used to prove
Theorem~\ref{t:main}, because there is no a priori degree bound on a
binomial in $I$ (cf. Example~\ref{e:deepBinomials}).

It appears that primary decomposition is not helpful for the problem
at hand. For example, the ideal $\<(x-y)(z-w)\> = \<x-y\>\cap\<z-w\>$
does not contain a binomial, even though its minimal primes are
generated by binomials.

Finally, the detection of monomials in a polynomial ideal is quite
simple using ideal quotients: an ideal $I\subset \QQ[x_1,\dots,x_n]$
contains a monomial if and only if
$((\cdots(I:x_1^\infty)\cdots):x_n^\infty) = (I:(x_1\cdots
x_n)^\infty) = \QQ[x_1,\dots,x_n]$.  The colon ideals $(I:x_i^\infty)$
can readily be computed with Gröbner
bases~\cite[Section~1.8.9]{singularbook}.  It was discovered several
times that extensions of this yield all monomials~(see
\cite{miller16:_findin}, \cite[Algorithm~4.4.2]{saito}, or
\cite[Tutorial~50]{kreuzer2005computational}).

\subsection*{Acknowledgement}
The authors are grateful to Bernd Sturmfels for suggesting to approach
the problem through tropical geometry and for constant
encouragement. The application of tropical geometry to binomial
containment was pointed out to the first author by Douglas Lind.
Alice Silverberg and Hendrik W. Lenstra provided valuable help
navigating the computational number theory literature, in particular,
locating a copy of~\cite{Ge}.

\section{Binomials in ideals}
As in the case of binomial ideals, it is more convenient to work not
only with the binomials in an ideal, but with the entire subspace they
generate.  Throughout this section, let $\kk$ be any fixed field and
denote by $S = \kk[x_1,\dots,x_n]$ the polynomial ring in $n$
indeterminates with coefficients in~$\kk$.
We occasionally use the notation $\xb^a := \prod_i x_i^{a_i}$ for $a = (a_1, \dotsc, a_n)\in\NN^n$.

\begin{defn}
Let $I\subset S$ be an ideal.  The \emph{binomial
part $\bin{I}$ of $I$} is the $\kk$-subspace of $I$ spanned by all
binomials in~$I$.
\end{defn}

\begin{prop}
	The binomial part of any ideal is a binomial ideal.
\end{prop}
\begin{proof}
	Let $I \subset S$ be an ideal and $B\subset I$ its
	binomial part. Then every element
	$b\in B$ is a linear combination of binomials.  Multiplying it with an
	arbitrary $f\in S$ yields some linear combination of
	monomial multiples of the binomials in~$b$.  Since $I$ is an ideal,
	those monomial multiples are contained in $B$ too and so is~$fb$. Thus $B$ is an ideal.
	Moreover, the ideal $B$ is binomial, since an ideal is in particular generated by any
	set that generates it as a vector space.
\end{proof}

By the same argument, a binomial ideal is as a vector space spanned by
the binomials it contains.  In particular, a binomial ideal equals its
binomial part.

We now discuss ring extensions in this context.  Denote by
$T = \kLaurx$ the Laurent polynomial ring
corresponding to~$S$. We extend the notion of $\bin{I}$ to this ring
in the natural way.
\begin{lemma}\label{l:saturate}
For any ideal $I \subset S$, it holds that $\bin{I T} = \bin{I}T$.  In
particular, $I$~contains a binomial if and only if the extension of
$I$ to $T$ contains a binomial.
Moreover, if $(I:x_1\dotsm x_n) = I$, then $\bin{I} = \bin{I T} \cap S$.
\end{lemma}
\begin{proof}
The inclusion ``$\supseteq$'' is clear because $I \subset I T$.  For
the other inclusion, note that any binomial in $I T$ can be multiplied
with a monomial to obtain a binomial in $I$.  The last statement
follows from the fact the hypothesis implies that $I = I T \cap S$.
\end{proof}

The binomial part of a proper ideal $I \subset T$ is determined by a
lattice $L \subset \ZZ^n$ and a homomorphism $\phi: L \to \kk^\units$
(called a \emph{partial character} in \cite{eisenbud96:_binom_ideal}).
According to \cite[Theorem~2.1]{eisenbud96:_binom_ideal}, the binomial
part of $I$ is the binomial ideal $\< \xb^m - \phi(m) \with m \in L\>$.
\begin{remark}\label{r:swallow}
For ideals in polynomial rings, a partial character is not a
sufficient data structure to store all binomials, for there typically
exist associated primes containing indeterminates.  When passing to
the Laurent ring, these associated primes are annihilated and many new
binomials may be created.  For example,
$\<x-y,x^2,xy,y^2\> \subset \kk[x,y]$ extends to the entire Laurent
ring $\kk[x^\pm,y^\pm]$, while for
\[
\<x^2-y^2, x^3-x^2y\>
\]
the information about the \emph{index two} lattice of binomials of
degree two is lost by the appearance of $x-y$, which generates the
extension to the Laurent ring.  However, Lemma~\ref{l:saturate}
guarantees that the extension to the Laurent ring only yields new
binomials if binomials are present in the original ideal.
\end{remark}

\begin{lemma}\label{lemma:field}
Let $\kk'/\kk$ be any field extension and
$T' \defas T \otimes_\kk \kk'$ be the Laurent polynomial ring over
$\kk'$.  Then, for any ideal $I\subset T$, it holds that
\[ \bin{IT'} = \bin{I}T'. \]
In particular,
\begin{itemize}
\item $I$ contains a binomial if and only if $IT'$
contains a binomial, and
\item $\bin{I} = \bin{I T'} \cap T$.
\end{itemize}
\end{lemma}

\begin{proof}
As above, the inclusion ``$\supseteq$'' is clear because
$I \subset I  T'$.  Moreover, the claim is clear if
$I  T' = T'$, so we may assume that $I$ is a proper ideal.
Suppose now that $I  T'$ contains the binomial $\xb^u-\lambda \xb^v$
with $u, v\in\ZZ^n\setminus\{0\}$, $\lambda \in \kk'$.  Then there
exists an expression
\begin{equation}\label{eq:linear}
\xb^u-\lambda \xb^v=\sum_i \lambda_{i}\xb^{v_i}f_i
\end{equation}
with $v_i\in\ZZ^n$, $\lambda_i\in\kk'$ and $f_i\in I$.  When
$u, v, v_i$ and $f_i$ are fixed, \eqref{eq:linear} can be interpreted
as a system of linear equations in the unknowns $\lambda$
and~$\lambda_i$.  This system has a solution over~$\kk'$, and because
its coefficients are in $\kk$, it also has a solution over~$\kk$.
Moreover, there is only one value possible for $\lambda$, because
otherwise $\xb^v \in I  T'$ and thus $I  T' = T'$.  Hence
$\lambda \in \kk$ and $\xb^u-\lambda \xb^v \in I$.  Every element of
$\bin{I}$ can be written as a linear combination of binomials of this
form, and the claim follows.

For the last claim, we only need to show the inclusion
$\bin{I} \supseteq \bin{I  T'} \cap T$.  Choose a $\kk$-basis
$\cB$ of $\bin{I}$. By the argument above, it is also a $\kk'$-basis
of $\bin{I  T'}$ and thus every binomial $b \in \bin{I  T'}$
has a unique expansion in this basis.  Hence, $b$ lies in $T$ if and
only if its coefficients in this expansion lie in $\kk$.  But the
latter implies that $b \in \bin{I}$.
\end{proof}

\begin{remark}\label{r:fieldContract}
Let $I \subset \QLaurx$ be an ideal.  Our algorithms usually construct
the binomial part of the extension $I\kLaurx$ to the Laurent ring with
coefficients in a finite extension $\kk$ of~$\QQ$.
Lemma~\ref{lemma:field} guarantees that this yields a determination of
the binomial part of an ideal $I\subset\QLaurx$, because
\[
\bin{I} = \bin{I\kLaurx} \cap \QLaurx.
\]
\end{remark}

\begin{example}\label{e:pullBackBin}
Remark~\ref{r:fieldContract} shows that the binomial part is preserved
when extending the coefficient field and then contracting back.  It is
not generally true that binomial parts survive contraction followed by
extension.  For example $\<x-\sqrt{2}\> \subset \QQ(\sqrt{2})[x]$
contracts to $\<x^2-2\> \subset \QQ[x]$ which in turn extends to
$\<x^2 - 2\> \subset \QQ(\sqrt{2})[x]$ by Lemma~\ref{lemma:field}.
\end{example}

\section{Reducing to the Artinian case via tropical geometry}
\label{s:reduceToArtinian}

Our eventual goal it to compute $\bin{I}$ for arbitrary
ideals~$I \subset \QQ[x_1,\dots,x_n]$.  In this section we use
tropical geometry which means that we have to work with the extension
of $I$ to the Laurent polynomial ring.  By Lemma~\ref{l:saturate} this
is sufficient to determine whether $\bin{I}$ is empty or not.
Moreover, if $(I : x_1\dotsm x_n) = I$ then our methods determine all
of~$\bin{I}$.

If $I$ is a prime ideal over an algebraically closed field, then
tropical geometry yields a complete answer: the ideal contains
binomials if and only if the tropical variety is contained in a
tropical hypersurface of a binomial i.e.~in an ordinary hyperplane
(Corollary~\ref{c:prime}).  In fact, one implication is immediate from
the following definitions.  The algebraically closedness assumption is
easy to relax, but if the ideal is not prime the tropical variety
alone does not reveal binomial containment as the following example
demonstrates.
\begin{example}
The principal ideal $\langle (x-1)(x-2)\rangle\subset\CC[x]$ has
tropical variety $\{0\}$.  It cannot contain a binomial, since such
binomial would have roots with different moduli, which binomials
cannot have.
\end{example}
However, expanding on the idea from the prime case we can use tropical
geometry to reduce binomial detection to the case of ideals with
Artinian quotients, which we call Artinian ideals for short.  The
results in this section hold for more general coefficient fields
than~$\QQ$.  To this end, let $\kk$ be a field and $\ol\kk$ its
algebraic closure.  The reader interested only in Theorem~\ref{t:main}
can mentally replace $\kk$ by~$\QQ$.  It is notationally convenient to
understand the Laurent ring $\kLaurx$ as the group ring~$\kk[\ZZ^n]$.
This is the ambient ring for this section.

\renewcommand{\RR}{\QQ}

\begin{defn}
\label{def:tropicalvariety}
Let $L$ be an integer lattice, $M$ the dual lattice and
$I\subset \kk[L]$ an ideal.  For $\omega\in\RR\otimes M$ the
\emph{initial form $\ini_\omega(f)$} of a polynomial
$f=\sum_{v}c_v\xb^v$ is the sum of terms $c_v\xb^v$ for which
$\langle \omega, v\rangle$ is maximal. For an ideal $I\subset \kk[L]$
the \emph{initial ideal of $I$ with respect to $\omega$} is
$\ini_\omega(I) = \langle \ini_\omega(f):f\in I\rangle$.  The
\emph{tropical variety} of $I$ is
\[
T(I)=\{\omega\in\RR\otimes M:\ini_\omega(I) \neq \kk[L]\}.
\]
\end{defn}

If $L=\ZZ^n$ and $I$ is homogeneous, the definition can be stated in
terms of initial ideals of homogeneous ideals in a polynomial ring.
Then $T(I)$ is the support of a subfan of the Gröbner fan
of~$I\cap\kk[\NN^n]$, a fan in $\RR^n$ that has one cone for each
initial ideal of~$I\cap\kk[\NN^n]$.  Since
$\ini_\omega(I) \neq \kk[L]$ can be decided by Gröbner bases, the
definition can be turned into an algorithm computing tropical
varieties~\cite{bjsst}.

If a polynomial $f$ is a binomial, then $T(\langle f\rangle)$ is a
hyperplane (or empty) and the Newton polytope of $f$ is a line segment
orthogonal to $T(\langle f\rangle)$. The inclusion
$T(\langle f\rangle)\supset T(I)$ for $f\in I$ implies that if $I$
contains a binomial $f$, then the Newton polytope of $f$ must be
perpendicular to $T(I)$.  Thus, if $I$ contains a binomial then also
$I\cap \kk[T(I)^\perp\cap L]$ contains a binomial.  The following
proposition extends this to all of~$\bin{I}$.
\begin{prop}\label{p:binproject}
Let $L$ be a lattice and $I\subset \kk[L]$ an ideal. Then
\[
\bin{I} = \bin{I \cap \kk[T(I)^\perp\cap L]}  \kk[L].
\]
\end{prop}
\begin{proof}
Let $f\in I$ be a binomial generator of the left-hand side. Then the
Newton polytope of $f$ is perpendicular to $T(I)$, meaning that
$\xb^uf\in I \cap \kk[T(I)^\perp\cap L]$ for some $u\in L$. Hence
$f=\xb^uf\xb^{-u}\in \bin{I \cap \kk[T(I)^\perp\cap L]} \kk[L]$.  The
other containment is clear since
$I \cap \kk[T(I)^\perp\cap L] \subset I$.
\end{proof}
The lattice $T(I)^\perp \cap L$ is a saturated lattice in $L$ and
therefore, after a multiplicative change of coordinates, we may assume
that
$(T(I)^\perp \cap L)\times \{0\}^{n-m}=\ZZ^{m}\times
\{0\}^{n-m}\subset\ZZ^n=L$ with $m=\dim(T(I)^\perp)$. Generators for
$I\cap \kk[T(I)^\perp\cap L]$ can then be computed by the elimination
$I\cap \kk[x_1^\pm,\dots,x_m^\pm]$.  This can be reduced to a Gröbner
basis computation in the polynomial ring by first passing to the
saturation $(I:(x_{m+1}\dots x_n)^\infty)$.  The following theorem
reduces the problem of deciding whether an ideal contains a binomial
to the case of Artinian ideals.

\begin{thm}\label{t:reduceArt}
Let $\kk$ be any field, $L$ an integer lattice, and $I\subset \kk[L]$
an ideal. Then $I\cap \kk[T(I)^\perp\cap L]$ is an Artinian ideal in
$\kk[T(I)^\perp\cap L]$.
\end{thm}
\begin{proof}
Let $L' \subseteq L$ be a saturated sublattice of $L$, and let
$\iota: L' \hookrightarrow L$ denote the inclusion map.  It gives rise
to a ring homomorphism $\kk[\iota]: \kk[L'] \to \kk[L]$, as well as to
a dual map $\iota^*_\QQ: L^* \otimes \QQ \to (L')^*\otimes \QQ$.  We
claim that
\begin{equation*}
\iota^*_\QQ(T(I)) = T(\kk[\iota]^{-1}(I)).
\end{equation*}
To see this, note that $\kk[\iota]$ induces a map
$\phi: (\kk^\units)^n \to (\kk^\units)^d$ of tori, where
$n = \rank(L)$ and $d = \rank(L')$. This map is monomial, because
$\kk[\iota]$ came from a map of lattices.  Hence
\cite[Corollary~3.2.13]{MaclaganSturmfels}
implies that $\iota^*_\QQ(T(I)) = T(\overline{\phi(V(I))})$.  On the
other hand, by classical elimination theory it holds that
$\overline{\phi(V(I))} = V(\kk[\iota]^{-1}(I))$, so taking the
tropical variety yields the claim.

Now we turn to the proof of the theorem. For this choose
$L' := T(I)^\perp \cap L$.  Then $\iota^*_\QQ(T(I)) = \set{0}$,
because restricting a linear map to its kernel yields zero.  On the
other hand, it clearly holds that
$\kk[\iota]^{-1}(I) = I \cap \kk[L']$, because $\kk[\iota]$ is an
inclusion. It follows that $\dim T(I \cap \kk[L']) = 0$.  Finally, by
the Bieri--Groves theorem ~\cite[Theorem~3.3.5]{MaclaganSturmfels},
this is also the dimension of the variety of $I \cap \kk[L']$, and
hence this ideal is Artinian.
\end{proof}

That our definition of tropical varieties is compatible with that in
\cite{MaclaganSturmfels} follows from the Fundamental Theorem of
Tropical Geometry~\cite[Theorem~3.2.3]{MaclaganSturmfels}.  We employ
\cite[Corollary~3.2.13]{MaclaganSturmfels} when $\kk$ is not
algebraically closed, which is possible since extending the field does
not affect the tropical varieties as they are defined via initial
ideals that are computable via Gr\"obner bases.
Similarly, if the field does not come with a non-trivial valuation
(which is the case here), one may extend it to the field of
generalized Puiseux series which has a non-trivial valuation.  See
also~\cite[Theorem~3.1.3]{MaclaganSturmfels}.

The preceding theorem allows us to determine when a \emph{prime} ideal
contains a binomial.
\begin{cor}\label{c:prime}
Let $L$ be an integer lattice and $I\subset \kk[L]$ an ideal.  If the
extension $I \ol{\kk}[L] \subset \ol\kk[L]$ of $I$ to the algebraic
closure is prime,
then $I$ contains a binomial if and only if $T(I)$ is contained in a
hyperplane, i.e.~$T(I)^\perp \neq \set{0}$.
\end{cor}
\begin{proof}
By Lemma~\ref{lemma:field} we may assume that $\kk = \ol{\kk}$.
Moreover, by Proposition~\ref{p:binproject} we can consider
$I' := I \cap \kk[T(I)^\perp \cap L]$ instead of $I$.  Now, if
$T(I)^\perp = \set{0}$, then $I' = \<0\>$ does not contain a binomial.
On the other hand, if $T(I)^\perp \neq \set{0}$, then $I'$ is a proper
Artinian ideal.  Hence, after choosing an identification
$\kk[T(I)^\perp \cap L] = \kLaurym$, $I'$ contains non-constant
univariate Laurent polynomials in each of the~$y_i$ and in particular
a Laurent polynomial~$f \in \kk[y_1^\pm]$.  Because $\kk$ is
algebraically closed, we can factor $f$ as
$f = c y_1^a\prod_j (y_1-\lambda_j)$with $c, \lambda_j \in \kk$ and
$a \in \ZZ$.  One factor is contained both in $\kk[T(I)^\perp \cap L]$
and in $I$ (because $I$ is prime) and hence in~$I'$. Thus
$\bin{I'} \neq \set{0}$.
\end{proof}

The intention is to apply Theorem~\ref{t:reduceArt} and
Proposition~\ref{p:binproject} to reduce the computation of $\bin{I}$
to the Artinian case for arbitrary~$I \subset \kk[L]$.  To proceed, we
must be able to compute the lattice $T(I)^\perp\cap L$.  We formulate
the following algorithms in the Laurent ring, but using saturations
the necessary computations can be carried out in a polynomial ring.

It is possible to either compute the entire Gröbner fan or to apply
the traversal strategy of~\cite{bjsst} even if $I$ is not homogeneous
to find $T(I)$, but both strategies have several drawbacks; a
problematic one being that $T(I)$ can easily consist of millions of
polyhedral cones.  For this reason we offer an approach to directly
compute $\Span(T(I))\subseteq\RR^n$.  We will make the assumption that
we have an algorithm with the following specification.
\begin{alg}[Tropical Curve]$ $\label{alg:tropicalcurve}\\
{\bf Input:} Generators for an ideal $I\subset\kk[L]$ defining $T(I)$ of dimension $1$.\\
{\bf Output:} The rays of $T(I)$.
\end{alg}
One such algorithm relying on tropical bases is presented
in~\cite{bjsst}.  Another one relying on projections and elimination
can be found in Andrew Chan's thesis~\cite{chanphd}.  We use
Algorithm~\ref{alg:tropicalcurve} to find a non-trivial vector in
$T(I)$ as follows:
\begin{alg}$ $\label{alg:primitive}\\
{\bf Input:} Generators for an ideal $I\subset\kk[L]$ with
$d = \dim(I)>0$.\\
{\bf Output:} A primitive vector in $T(I)\setminus\{0\}$.
\begin{enumerate}
\item\label{it:gen} Choose $d-1$ polynomials
$u_1,\dotsc, u_{d-1} \in \Span_\kk\set{1, x_1,\dots,x_n}$ so that\\
$\dim (I + \<u_1, \dotsc,u_{d-1}\>) = 1$.
\item Compute $T(I + \<u_1, \dotsc,u_{d-1}\>)$ using Algorithm~\ref{alg:tropicalcurve}.
\item Return a primitive generator for one of the rays of $T(I + \<u_1, \dotsc,u_{d-1}\>)$.
\end{enumerate}
\end{alg}
The returned vector is indeed contained in $T(I)$, because
$T(I + \<u_1, \dotsc,u_{d-1}\>) \subseteq T(I)$.
\begin{remark}
The dimension condition in the first step holds for a Zariski open subset
of $\Span_\kk\set{1, x_1,\dots,x_n}$. Therefore these polynomials could be picked at random and checked
to satisfy the dimension condition.  There is also a deterministic way
using stable intersections and rational functions as coefficients.  Building on techniques similar to
\cite[Lemma~3.3]{jensenyustable} one can always find
suitable univariate linear polynomials.
\end{remark}

We can now state the algorithm to compute~$\Span(T(I))$.
\begin{alg}\label{alg:tropspan}$ $\\
\textbf{Input:} Generators for an ideal~$I\subset\kLaurx$.\\
\textbf{Output:} A vector space basis of~$\textup{span}(T(I))$.
\begin{enumerate}
\item Let $d := \dim(I)$.
\item If $\dim(I)=0$, then return the basis $\emptyset$ for $\{0\}$.
\item Compute a primitive vector $v\in T(I)$ using
Algorithm~\ref{alg:primitive}.
\item \label{it:changeCoord} Compute an invertible matrix
$M\in \ZZ^{n\times n}$ such that $Mv = (0,\dots,0,1)$.
\item Compute generators of the ideal $I' = \phi(I)$ where
$\phi : \kLaurx \to \kk[y_1^\pm,\dots,y_n^\pm]$ is the multiplicative
coordinate change induced by $y_i = \prod_{j=1}^n x_j^{M_{ij}}$.
\item Compute $J'=I'\cap\kk[y_1^\pm,\dots,y_{n-1}^\pm]$.
\item Recursively compute generators $U$ for
$\Span(T(J'))\subseteq\RR^{n-1}$.
\item Return $\set{v}\cup\set{M^{-1}(u\oplus(0)): u\in U}$.
\end{enumerate}
\end{alg}

\section{The Artinian case}
\label{s:ArtinianAlgorithm}
The proof of Theorem~\ref{t:main} is finished once we describe how to
compute ideal generators for the binomial part $\bin{I}$ of an ideal
$I \subset T$ with Artinian quotient~$T/I$, where $T = \QLaurym$ as
above.  For $1 \leq i \leq m $, let $M_i: T/I \to T/I$ denote the
linear endomorphism induced by multiplication with $y_i$.  With
$\ell = \dim_\QQ T/I$ let $\kk$ be the finite extension of $\QQ$ which
contains the $\ell$-th roots of the determinants of the~$M_i$.  Define
$M'_i = M_i / \sqrt[\ell]{\det M_i}$.  By Remark~\ref{r:fieldContract}
it suffices to determine the binomial part of the extension
$I\kLaurym$.
This computation can be translated into a membership problem in the
multiplicative group generated by the~$M_i$.
\begin{prop}\label{p:MakeCoeffOne}
Let $e \in \ZZ^m$. There exists a $\lambda \in \kk$ such that
$\yb^e - \lambda \in I$ if and only if
\begin{equation}\label{eq:commmatrix}
\prod_{i=1}^m (M'_i)^{e_i} = \Id_{T/I}.
\end{equation}
\end{prop}
\begin{proof}
The binomial $\yb^e - \lambda$ is contained in $I$ if and only if
$\prod_i (M_i)^{e_i} = \lambda\Id_{T/I}$.  Taking determinants of both
sides yields that in this case $\lambda^\ell = \prod_i (\det M_i)^{e_i}$.
So the claim follows from the definition of the $M_i'$.
\end{proof}
The matrices $M_i'$ commute, are invertible and have entries in a
finite extension of~$\QQ$.  In this situation, \cite[Theorem 1.2,
Section~6.4]{commmatrix} gives an algorithm that, for any matrices
$M_1', \dots M_m'$ with entries in a number field, computes a basis
for the lattice of exponents $e \in \ZZ^m$
satisfying~\eqref{eq:commmatrix}.
A more general version is
\cite[Algorithm~8.3]{lenstra2015algorithms}. Both rely on the LLL
lattice basis reduction algorithm.

\begin{remark}\label{r:non-comm}
The commutativity of the matrices is key for algorithmic treatment.
For general matrix semigroups, several problems are known to be
algorithmically undecidable (see for example the table in the end
of~\cite{halava1997decidableSurvey}).  In particular, there is no
Turing machine program that can decide whether there is a relation
among given $(3\times 3)$ matrices~\cite{klarner1991undecidability}.
It is also undecidable if a semigroup generated by eight $(3\times 3)$
integer matrices contains the zero
matrix~\cite{paterson1970unsolvability}.  This result of Paterson is
an important tool to prove other undecidability results.  For
invertible matrices, group membership is unsolvable for matrices of
format $(4\times 4)$ and larger~\cite{mihailova1958occurrence}.  Our
methods are therefore not directly applicable to polynomials in
non-commutative variables.
\end{remark}

Finally, the binomial part of the radical of an Artinian ideal
$I \subset \kLaurym$ can be computed without first
computing the radical itself.
\begin{prop}
For $e \in \ZZ^m$ there exists a $\lambda \in \kk$ such that
$\yb^e - \lambda \in \sqrt{I}$ if and only if
$\prod_{i=1}^m (M_i)^{e_i}$ has only one eigenvalue over the algebraic
closure~$\ol\kk$.  In this case,
$\lambda = \prod_i (\det M_i)^{e_i/\ell}$.
\end{prop}
\begin{proof}
Let $M = \prod_{i=1}^m (M_i)^{e_i}$.  Some power of $\yb^e - \lambda$
lies in $I$ if and only if $M - \lambda\Id_{T/I}$ is nilpotent.
Choose a basis such that $M$ is upper triangular.  Then
$M - \lambda\Id_{T/I}$ is nilpotent if and only if all entries on the
main diagonal of $M$ equal $\lambda$.  This equivalent to $M$ having
$\lambda$ as its sole eigenvalue.  Computing the determinant of $M$
then yields the claimed expression for~$\lambda$.
\end{proof}

Let $V \subseteq T/I$ be the direct sum of all eigenspaces of~$M_1$.
Then the restriction $M_1|_V$ of $M_1$ to $V$ is diagonalizable.
Since the $M_i$ commute, the same holds for all other $M_i|_V$.
Moreover, the set of eigenvalues of $M_i|_V$ equals the set of
eigenvalues of $M_i$ for each~$i$.

As above, set $M_i' = M_i / \sqrt[\ell]{\det M_i}$, where
$\ell = \dim_\QQ T/I$.  Let $L \subseteq \ZZ^m$ be the lattice of
exponents $e$ that satisfy
\[ \prod_{i=1}^m (M_i'|_V)^{e_i} = \Id_V. \]
Then $L$ can be computed with the algorithm in \cite{commmatrix}.  It
is clear from the observation above that for each $e \in L$, the
matrix $\prod_i (M_i)^{e_i}$ has only one eigenvalue. So $L$ contains
precisely the exponents of the binomials in the radical of $I$.

\section{Algorithm}
\label{sec:complete_algorithm}
To facilitate an implementation of the methods in this paper we
formulate the complete algorithm.  In this formulation the algorithm
returns generators of the Laurent extension $\bin{I}\QLaurx$ which, by
means of Lemma~\ref{l:saturate}, yields the desired algorithm for
Theorem~\ref{t:main}.

\begin{alg}\label{alg:complete}$ $\\
\textbf{Input:} Generators $f_1,\dots,f_s$ for an ideal $I \subset \QQ[x_1,\dots,x_n]$\\
\textbf{Output:} Generators of $\bin{I}\QLaurx$.

\begin{enumerate}
\item Let $J= \<f_1,\dots,f_s\> \subset \QLaurx$ be the Laurent
extension of~$I$.
\item Compute the orthogonal complement $T(J)^\perp$ of the tropical
variety of~$J$ by Algorithm~\ref{alg:tropspan}.
\item Compute a basis $\{b^{(1)},\dots, b^{(m)}\}$ of the integer
lattice $L = T(J)^\perp \cap \ZZ^n$.  Let $\QLaurym$
with $y_j = \xb^{b^{(j)}}$ be the resulting Laurent
polynomial subring of $\QLaurx$.
\item Compute $K = J\cap \QQ[L] \subset \QLaurym$ as the preimage of
$J$ under the inclusion of $\QQ$-algebras $\QLaurym \to \QLaurx$.
\item Pick a basis of the finite-dimensional $\QQ$-algebra
$\QLaurym/K$ and compute the matrix representations
$M_i$ of the linear maps given by multiplication with $y_i$.
\item Construct a number field $\kk$ that contains all $\ell$-th roots
of the determinants of the~$M_i$, where $\ell := \dim_\QQ \QLaurym/K$.
Compute $M_i' = M_i/\sqrt[\ell]{\det M_i}$
\item Compute a basis $\{c^{(1)},\dots, c^{(t)}\}$ of the lattice
$E\subset \ZZ^m$ of exponents satisfying~\eqref{eq:commmatrix}, for
example using the algorithm given in~\cite{commmatrix}.
\item For each $i=1,\dots,t$ compute $\lambda_i$, such that
$\yb^{c^{(i)}} - \lambda_i \in K$, for example by using
$\prod_{j}(M_j)^{c^{(i)}_j} = \lambda_i \Id$.  Then
$\bin{K} = \<\yb^{c^{(1)}} - \lambda_1,\dots, \yb^{c^{(t)}} - \lambda_t\> \subset
\QLaurym$.
\item Return generators of $\bin{J}$ by substituting $\xb^{b^{(j)}}$
for~$y_j$ in the generators of $\bin{K}$.
\end{enumerate}
\end{alg}

\renewcommand*{\bibfont}{\footnotesize}
\printbibliography

\end{document}
